\documentclass[12pt]{amsart}
\usepackage{amsmath, amssymb}
\usepackage{mathrsfs}
\newcommand{\no}[1]{#1}
\renewcommand{\no}[1]{}
\no{\usepackage{times}\usepackage[subscriptcorrection, slantedGreek, nofontinfo]{mtpro}
\renewcommand{\Delta}{\upDelta}}
\usepackage{color}
\usepackage{enumerate,comment}

\date{\today}
\newtheorem{theorem}{Theorem}[section]
\newtheorem{prop}{Proposition}[section]
\newtheorem{lem}{Lemma}[section]

\newtheorem{corollary}{Corollary}[section]

\theoremstyle{remark}

\newcommand{\bel}{\begin{equation} \label}
\newcommand{\ee}{\end{equation}}

\newcommand{\pd}{\partial}

\newcommand{\R}{{\mathbb R}}
\newcommand{\N}{{\mathbb N}}

\def\phi {\varphi}
\def\epsilon {\varepsilon}

\renewcommand{\leq}{\leqslant}
\renewcommand{\geq}{\geqslant}

\def\beq{\begin{equation}}
\def\eeq{\end{equation}}
\newcommand{\bea}{\begin{eqnarray}}
\newcommand{\eea}{\end{eqnarray}}
\newcommand{\beas}{\begin{eqnarray*}}
\newcommand{\eeas}{\end{eqnarray*}}

\providecommand{\norm}[1]{\left\lVert#1\right\rVert}

\numberwithin{equation}{section}

\title[Rigidity of inverse problems for semlinear elliptic equations]{Rigidity of inverse problems for nonlinear elliptic equations on manifolds}

\begin{document}
\author[Feizmohammadi]{Ali Feizmohammadi}

\address{A. Feizmohammadi, Department of mathematics, University of Toronto, Mississauga, ON L5L 1C6, Canada}

\email{ali.feizmohammadi@utoronto.ca}

\author[Kian]{Yavar Kian}

\address
        {Y. Kian, Aix Marseille Univ\\ 
        Universit\'e de Toulon, CNRS\\
        CPT, Marseille, France}

\email{yavar.kian@univ-amu.fr}

\author[Oksanen]{Lauri Oksanen}

\address
        {L. Oksanen, Department of Mathematics and Statistics, University of Helsinki,\\
				P.O 68, 00014 University of Helsinki, Finland}

\email{lauri.oksanen@helsinki.fi}
\begin{abstract}

We consider the inverse problem of determining coefficients appearing in semilinear elliptic equations stated on Riemannian manifolds with boundary given the knowledge of the associated Dirichlet-to-Neumann map. We begin with a negative answer to this problem. Owing to this obstruction, we consider a new formulation of our inverse problem in terms of a rigidity problem. Precisely, we consider cases where the Dirichlet-to-Neumann map of a semilinear equation coincides with the one of a linear equation and ask whether this implies that the equation must indeed be linear. We give a positive answer to this rigidity problem under some assumptions imposed to the Riemannian manifold and to the semilinear term under consideration.

\medskip
\noindent

\end{abstract}

\maketitle


\section{Introduction}
Let $(M,g)$ be a smooth compact and connected Riemannian manifold of dimension $n\geq2$ with boundary and fix $\alpha\in(0,1)$.  We consider the semilinear elliptic boundary value problem
\bel{eq1}
\left\{
\begin{array}{ll}
-\Delta_g u+a(x,u)=0  & \mbox{in}\ M ,\\
u= f&\mbox{on}\ \partial M,
\end{array}
\right.
\ee
with $a\in \mathcal C^2(M\times\R)$ subject to any Dirichlet data $f\in\mathcal C^{2+\alpha}(\partial M)$. In general, existence and uniqueness of solutions to equation \eqref{eq1} is not guaranteed. A typical assumption that makes the latter boundary value problem globally well-posed is to assume that the nonlinear term $a$ satisfies the conditions
\bel{cond1} a(x,0)=0,\quad x\in M,\ee
\bel{cond3}\partial_\mu a(x,\mu)\geq0,\quad x\in M,\ \mu\in\R.\ee
 Under this assumption, equation \eqref{eq1} admits a unique solution $u\in C^{2+\alpha}(M)$. Moreover, the dependence of $u$ on $f$ is continuous. For the sake of motivating our main results, let us assume for the remainder of this section that \eqref{cond1} and \eqref{cond3} are satisfied. We define the Dirichlet-to-Neumann map associated to \eqref{eq1} via
$$ \Lambda_a(f) = \pd_\nu u|_{\pd M},$$
where $\nu$ is the exterior normal vector field on $\pd M$. The central topic of this paper is the following inverse problem:
\medskip

\noindent{\bf IP:} \textit{Does the knowledge of $\Lambda_a$ uniquely determine $a$ on $M\times \R$? In other words, if $a_1,a_2\in \mathcal C^{2}(M\times \R)$ satisfy \eqref{cond1}, \eqref{cond3} and if $\Lambda_{a_1}=\Lambda_{a_2}$, does it follow that $a_1=a_2$ on $M\times \R$?} \\

The inverse problem IP can be seen as the determination of the nonlinear law for problems of minimization of nonlinear energy or some class of reaction diffusion equations with stationary solutions. This problem has received considerable attention among the mathematical community. One of the first results about the determination of this class of nonlinear terms goes back the work of \cite{IS} when $n=3$ and  \cite{IN} when $n=2$. Some extended results have been obtained by \cite{Is1,Is2} for semilinear terms depending only on the solutions. We mention also the recent development of problem IP for nonlinear terms $a$ subjected to the analyticity of the map $u\mapsto a(\cdot,u)$ where the authors used the higher order linearization technique, initiated by \cite{KLU}, for solving the inverse problem IP by exploiting  the nonlinear interaction. In that category of results, without being exhaustive, we can mention the works of \cite{FO20,KU0,KU,LLLS}.

In this article we will consider the determination of general semilinear terms without the analyticity condition considered by \cite{FO20,KU0,KU,LLLS}. Namely, in a similar way to \cite{IN,IS}, we consider the problem IP for general semilinear terms $a\in \mathcal C^2(M\times\R)$  only subjected to conditions that guarantee the well-posedness of \eqref{eq1}. In order to discuss this problem, let us first introduce the following notations.  We write $u(x;f)$ for the solution to \eqref{eq1} subject to Dirichlet data $f$ and define 
 
 \begin{equation}
 	\label{u_+}
 	u_\star(x)=\inf \{ u(x;f)\,:\, f\in C^{2+\alpha}(\pd M)\}\quad x\in M,
 \end{equation}
and 
\begin{equation}
	\label{u^+}
	u^\star(x)=\sup \{ u(x;f)\,:\, f\in C^{2+\alpha}(\pd M)\}\quad x\in M.
\end{equation}
We also define
\begin{equation}
	\label{E}
	E= \{(x,\mu)\,:\, x\in M,\quad \mu\in (u_\star(x),u^\star(x))\}.
\end{equation}

When $(M,g)$ is a Euclidean domain, we recall the following result that is due to Isakov and Sylvester \cite{IS} when $n=3$ and to Isakov and Nachman \cite{IN} when $n=2$ (we have relaxed the regularity assumption on $a$ here).  

\begin{theorem}[see \cite{IN,IS}]
Let $(M,g)$ be a bounded Euclidean domain with a smooth boundary. Let $a_1,a_2\in \mathcal C^2(M)$ satisfy \eqref{cond1} and \eqref{cond3}. If, 
$$ \Lambda_{a_1}(f) = \Lambda_{a_2}(f) \quad \forall\, f\in C^{2+\alpha}(\pd M),$$
it follows that $E_1=E_2=:E$, and that 
\begin{equation}\label{result}
	a_1(x,\mu)=a_2(x,\mu), \quad \forall\, (x,\mu)\in E.
\end{equation}
\end{theorem}

The latter theorem establishes equality of $a_1$ and $a_2$ on the \textit{a priori} unknown and abstract set $E$. This is of course the best result that one may hope for, but the following question arises: To what an extent can we describe the set of interior values of solutions to \eqref{eq1} at a fixed point $x$ in the interior of the manifold $M$? In \cite{IN,IS} it was proved that assuming the stronger assumption 
$$ \sup_{\mu \in \R} |\pd_\mu a(x,\mu)| \in L^p(M), \quad \text{for some $p>1$},$$
we have $E= M\times \R$. Our first result shows that the global Lipschitz assumption on $a$ may in fact be close to optimal, in order to achieve $E=M\times \R$. Precisely, we prove the following theorem in Section~\ref{sec_comp}.

\begin{theorem}
	\label{comp_thm}
	Let $(M,g)$ be a bounded Euclidean domain in $\R^n$, $n\geq 2$, with a smooth boundary. Let $a\in \mathcal C^2(M\times \R)$ satisfy \eqref{cond1} and \eqref{cond3}. Assume also that 
	\begin{equation}
		\label{a_bound_1}
		F(\mu)\leq a(x,\mu), \quad \forall (x,\mu)\in M\times [0,\infty),
	\end{equation}
	and 
	\begin{equation}
		\label{a_bound_2}
		a(x,\mu) \leq F(\mu) \quad \forall (x,\mu)\in M\times (-\infty,0],
	\end{equation}
	for some function $F\in \mathcal C^2(\R)$ that satisfies
	\begin{equation}\label{F_1} F(0)=0,\quad \text{and}\quad F'(\mu)\geq 0 \quad \forall\, \mu \in \R,\end{equation}
	and 
	\begin{equation}\label{F_2}|F(\mu)| \geq \delta\,|\mu|^{1+\epsilon}\quad \text{for some $\delta,\epsilon>0$.}
	\end{equation}
	Let $h>0$ and define $ U_h= \{x\in M\,:\, \textrm{dist}(x,\pd M)\geq h\}.$
	Then, 
	\begin{equation}\label{up} |u(x;f)| \leq C_h, \qquad \forall\, f\in C^{2+\alpha}(\pd M),\end{equation}
	for some constant $C_h>0$ that is independent of $f$. Here, $u(x;f)$ is the solution to \eqref{eq1} subject to Dirichlet data $f$.
\end{theorem} 
Note, for example, that for any integer $p>1$ and any $\{b_k\}_{k=1}^p \subset \mathcal C^2(M;(0,\infty))$, the function
$$ a(x,\mu) = \sum_{k=1}^{p} b_k(x) \mu^{2k-1},$$
satisfies the hypotheses of Theorem~\ref{comp_thm}. To the best of our knowledge, Theorem \ref{comp_thm}  gives the first negative answer to the inverse problem IP for some general class of semilinear terms $a$. Such obstruction was mentioned in \cite{IS}, but the authors considered only a counterexample on an interval in dimension 1 and for a specific power type nonlinearity. Not only Theorem \ref{comp_thm} describes an obstruction to problem IP on arbitrary domains in any dimension $n\geq2$ but it can also be applied to a general class of semilinear terms subjected to the conditions \eqref{a_bound_1}--\eqref{a_bound_2}.

 In order to explain why Theorem \ref{comp_thm}  gives a negative answer to the inverse problem IP, let us denote by Int$(M)$ the interior of the manifold $M$, and for any $x\in$Int$(M)$ let us consider $h_x=\textrm{dist}(x,\pd M)$. Now consider $a_j\in \mathcal C^2(M\times \R)$, $j=1,2$, with $a_1\neq a_2$, satisfying the condition \eqref{a_bound_1}--\eqref{a_bound_2} for $a=a_j$ and $F_\pm$ satisfying the conditions \eqref{F_1}-\eqref{F_2}. Fix also, for any $x\in$Int$(M)$, $C_x$ the constant $C_h$ of the estimate \eqref{up} with $h=h_x$, and assume that 
$$a_1(x,\mu)=a_2(x,\mu),\quad (x,\mu)\in \pd M\times\R\textrm{ or } x\in\textrm{Int}(M),\ \mu\in[-C_x,C_x].$$
Then, estimate \eqref{up} implies that for any $u\in C^{2+\alpha}( M)$ we have
$$(-\Delta u(x)+a_1(x,u(x))=0,\ x\in M)\Longleftrightarrow(-\Delta u(x)+a_2(x,u(x))=0,\ x\in M).$$
As a consequence of this property we have $\Lambda_{a_1} = \Lambda_{a_2}$ but $a_1\neq a_2$. Therefore, Theorem \ref{comp_thm} introduces an obstruction to the determination of a general semilinear terms from boundary data and it confirms that there is no hope to solve inverse problem IP for some large class of semilinear terms.

Owing to this obstruction, a natural follow up direction is to try to recover certain information about the nonlinear term. One example is a rigidity question that can be posed as follows:\\
\noindent{\bf IP1:} \textit{If the boundary data $\Lambda_a$ coincides with the one of a linear term $\tilde{a}$, does it mean that $a$ is linear?}\\

The above rigidity problem will indeed be the second main topic of this article (in addition to the general obstruction for the determination of a semilinear term exhibited by Theorem \ref{comp_thm}). We will focus our attention on the resolution of the problem IP1 with different class of  assumptions imposed to the manifold $(M,g)$ or the nonlinearity $a(x,u)$. 

Before closing this section, we mention a few examples of rigidity results in the context of inverse problems for linear PDEs. For instance, in the context of the Electrical Impedance Tomography (EIT), the well-known Calder\'{o}n problem asks whether it is possible to determine the conductivity of a medium from boundary measurements of electrical voltage and current flux. If the conductivity is assumed to be isotropic, the first result in this direction was a rigidity result due to Calder\'{o}n \cite{Calderon}. If the conductivity is anisotropic, analogous rigidity results are also available, see \cite{CF,KLS,Shar}.

\section{Main results for problem IP1}
In this article we will study the inverse problem IP1 under different assumptions imposed to the manifold $(M,g)$.
 Assuming that the condition \eqref{cond1} and
\bel{cond2}v\in H^1_0(M),\  \int_M\left\langle \nabla^gv,\nabla^gv\right\rangle_gdV_g(x)+ \int_M\partial_\mu a(x,0)|v|^2dV_g(x)=0\Longrightarrow v\equiv0\ee
are fulfilled, we prove in Proposition \ref{p1} that problem \eqref{eq1} is well-posed for sufficiently small Dirichlet data $f$. Namely, there exists $\lambda>0$ depending on $a$ and $(M,g)$ such that for all $f\in B_{\lambda}:=\{f\in C^{2+\alpha}(\partial M):\ \norm{f}_{ C^{2+\alpha}(\partial M)}<\lambda\}$ the problem \eqref{eq1} admits a unique solution $u_f\in C^{2+\alpha}( M)$ such that the map $f\mapsto u_f$ is continuous from $B_{\lambda}$ to $C^{2+\alpha}( M)$. Therefore, we can set
\bel{lambda}(0,+\infty]\ni\lambda_1=\sup\{\lambda>0:\ \textrm{for all $f\in B_{\lambda}$ problem \eqref{eq1} is well-posed in $ C^{2+\alpha}( M)$}\}\ee
and we can define the DN map associated to \eqref{eq1} as follows
$$\Lambda_a:B_{\lambda_1}\ni f\mapsto \partial_\nu u|_{\partial M}.$$
Note that here we have $B_{+\infty}=C^{2+\alpha}(\partial M)$. We will state our results for this class of nonlinear terms.

We start by considering a first result stated with a general Riemannian manifold for nonlinear terms satisfying the conditions \eqref{cond1}-\eqref{cond3}. Note that for this class of nonlinear terms $\lambda_1=+\infty$ and for all $f\in C^{2+\alpha}(\partial M)$ problem \eqref{eq1} will be well-posed and the DN map $\Lambda_a$ can be defined on $C^{2+\alpha}(\partial M)$ or equivalently on $B_{+\infty}$. For this class of nonlinear terms we consider the DN map $\Lambda_a$ as a map defined on $C^{2+\alpha}(\partial M)$ and we can prove the following result.

\begin{theorem}
\label{t1}
Let $a\in \mathcal C^2(M\times\R)$ satisfy condition \eqref{cond1} and \eqref{cond3}. Consider also $f\in C^{2+\alpha}(\partial M)$
a Dirichlet data satisfying 
\bel{t1a}\inf_{x\in \partial M}f(x)=c>0\ee
and the family $(h_\mu)_{\mu\in\R}$ of elements $h_\mu\in C^{2+\alpha}(\partial M)$.
Then, for any $\mu_1>0$, the condition 
\bel{t1b} \Lambda_a(\mu f)=\Lambda_0(h_\mu),\quad \mu\in(-\mu_1,\mu_1)\ee
implies
\bel{t1c} a(x,\mu )=0,\quad (x,\mu)\in M\times(-c\mu_1,c\mu_1),\ee
and, for any   $x_0\in\partial M$, we have
\bel{t1dd} h_\mu(x)=h_\mu(x_0)+\mu (f(x)-f(x_0)),\quad (x,\mu)\in \partial M\times(-\mu_1,\mu_1).\ee
\end{theorem}

As a direct consequence of Theorem \ref{t1} we can prove the following.

\begin{corollary}\label{c1} Let the conditions of Theorem \ref{t1} be fulfilled. Then 
$$\Lambda_a(\mu f)=\Lambda_0(\mu f),\quad \mu\in(-\mu_1,\mu_1)$$
implies \eqref{t1c}.\end{corollary}

In order to extend this result we need to impose more assumptions on the class of nonlinear terms under considerations. For semilinear terms that are  of separated variables, we obtain the following extension of Theorem \ref{t1}.
\begin{theorem}
	\label{thm_analytic}
	Let $(M,g)$ be a smooth compact Riemannian manifold of dimension $n\geq 2$ with boundary. Suppose that $a\in C^{\infty}(M\times \R)$ is of the form $a(x,s)= q(x) F(s)$ for some smooth functions $q$ and $F$ with $F(0)=F'(0)=0$ and such that $F^{(m)}(0)\neq 0$ for some $m\in \N$.
If 
\bel{DN_a} \Lambda_{a}(f) = \Lambda_0(f),\ee
for all $f\in C^{2,\alpha}(\pd M)$ in a neighborhood of the origin, then, 
$$ a(x,s) = 0 \quad \forall \, (x,z) \in M\times \R.$$
\end{theorem}

Next, if the nonlinear term $a(x,u)$ is even and analytic with respect to $u$, we have the following result.
 
\begin{theorem}
	\label{thm_analytic_2}
	Let $(M,g)$ be a smooth compact Riemannian manifold of dimension $n\geq 2$ with boundary. Suppose that $a\in C^{\infty}(M\times \R)$ is real-analytic with respect to $s$ in the sense of taking values in $\mathcal C^\alpha(M)$. Assume that $a(x,0)=\pd_s a(x,0)=0$ for all $x\in M$ and also that $a(x,s)=a(x,-s)$ for all $x\in M$ and $s\in \R$. If 
	\bel{DN_a} \Lambda_{a}(f) = \Lambda_0(f),\ee
	for all $f\in C^{2,\alpha}(\pd M)$ in a neighborhood of the origin, then, 
	$$ a(x,s) = 0 \quad \forall \, (x,s) \in M\times \R.$$
\end{theorem}

We obtain also several improvements of Theorem \ref{t1} by imposing some geometrical assumptions to the Riemmanian manifold $(M,g)$.
Namely, we assume that $(M,g)$  is a  {\em transversally anisotropic} which means that $M \subset \R\times M_0$ and $g= dt^2 + g_0$ where $(M_0,g_0)$ is a compact Riemannian manifold with boundary.

\begin{theorem}
\label{t2}
Let $(M,g)$ be transversally anisotropic of dimension $n\geq 3$ and assume that the transversal manifold $(M_0,g_0)$ is nontrapping. 
We fix $a\in \mathcal C^\infty(M\times\R)$ satisfying condition \eqref{cond1} and such that
 \bel{t2a} \partial_\mu a(x,0)\geq 0.\ee
Then, \eqref{cond2} is fulfilled and there exists $\lambda\in(0,\lambda_1]$ such that the condition
 \bel{t2b} \Lambda_af=\Lambda_0f,\quad f\in B_\lambda \ee 
implies
\bel{t2c} a(x,\mu )=0,\quad (x,\mu)\in M\times[-\lambda,\lambda].\ee
\end{theorem}

As an extension of these results, we can also consider the situation where the DN map of an expression $a(x,u)$ coincide with one of and expression  $\tilde{a}(x,u)$ linear in $u$. These results can be stated as follows.

\begin{theorem}
\label{t3}
Let $(M,g)$ be transversally anisotropic of dimension $n\geq 3$ and assume that the transversal manifold $(M_0,g_0)$ is simple. 
We fix $a_j\in \mathcal C^\infty(M\times\R)$, $j=1,2$, satisfying condition \eqref{cond1} and \eqref{cond2}, such that
 \bel{t3a} a_1(x,\mu)=q(x)\mu,\quad (x,\mu)\in M\times \R,\ee
with $q\in L^\infty(M)$ a non-negative function.
Then for any $\lambda\in(0,\lambda_1]$ the condition
 \bel{t3b} \Lambda_{a_1}f=\Lambda_{a_2}f,\quad f\in B_\lambda \ee 
implies that there exists $c>0$ depending on $q$ and $M$ such that
\bel{t2c} a_2(x,\mu )=q(x)\mu,\quad (x,\mu)\in M\times[-c\lambda,c\lambda].\ee
\end{theorem}

\begin{theorem}
\label{t4}
Let $(M,g)$ be transversally anisotropic of dimension $n\geq 3$ and assume that the transversal manifold $(M_0,g_0)$ contains a dense set of points $p\in M$ with a non-self-intersecting maximal geodesic 
through $p$ that contains no conjugate points to $p$. 
We fix $a_j\in \mathcal C^\infty(M\times\R)$, $j=1,2$, satisfying condition \eqref{cond1}, \eqref{cond2},  \eqref{t3a}  and such that
\bel{t4aa} \partial_\mu a_2(x,0)\geq0,\quad x\in M.\ee
Then,  there exists $\lambda\in(0,\lambda_1]$ such that the condition
 \bel{t4b} \Lambda_{a_1}f=\Lambda_{a_2}f,\quad f\in B_\lambda \ee 
implies that there exists $c>0$ depending only on $M$ and $\partial_\mu a_2(\cdot,0)$ such that
\bel{t4c} a_2(x,\mu )=\partial_\mu a_2(x,0)\mu,\quad (x,\mu)\in M\times[-c\lambda,c\lambda].\ee
\end{theorem}

Let us observe that Theorem \ref{t1} is stated for general Riemannian manifolds with semilinear terms subjected to the natural conditions \eqref{cond1} and \eqref{cond3} guaranteeing the well-posedness of \eqref{eq1}. In addition, the excitations applied to the systems, associated with the Dirichlet boundary conditions $\mu f$ and $h_\mu$, are not necessary the same. In contrast to most related results, not only we remove the analyticity condition imposed to the map $\R\ni u\mapsto a(\cdot,u)$ but we also consider general manifolds not restricted to transversally anisotropic manifolds. In that sense, the results of Theorem \ref{t1}  give a better understanding of the rigidity problem IP1 in a rather general context. As a direct consequence of Theorem \ref{t1}, we obtain Corollary \ref{c1} stated with more canonical class of data. Note that even the determination of a coefficient of a linear equation in this context is still an open problem.

Let us remark that Theorem \ref{thm_analytic} gives an extension of Theorem \ref{t1} by replacing the positivity condition \eqref{cond3} by the more general condition \eqref{cond2} for a semilinear terms  of separated variables. In the same way, in Theorem \ref{thm_analytic_2} we consider semilinear terms $a(x,u)$ that are analytic and even in $u$ to which we do not impose \eqref{cond3}.
Note that, the results of Theorem \ref{thm_analytic} and \ref{thm_analytic_2} are still stated on a general Riemannian manifold and the well-posedness of problem \eqref{eq1} for semilinear terms satisfying conditions  \eqref{cond1}-\eqref{cond2} will be proved in Proposition \ref{p1}.

The results of Theorem  \ref{t2}, \ref{t3} and \ref{t4} give several improvements of Theorem \ref{t1} when the Riemannian manifold $(M,g)$ is assumed to be  transversally anisotropic. Namely,  in Theorem \ref{t2} we show that the result of Theorem \ref{t1} is still locally true when the condition \eqref{cond3} is replaced by the more general condition \eqref{t2a}, while in Theorem \ref{t3} and \ref{t4} we study the rigidity issue IP1 when the linear equation takes the more general form $-\Delta_gu+qu=0$ on $M$ with $q\in L^\infty(M)$.

This article is organized as follows.  In Section 3, we show the obstruction to problem IP stated in Theorem~\ref{comp_thm}. Section 4 will be devoted to the well-posdness of problem \eqref{eq1} for semilinear terms satisfying conditions \eqref{cond1}-\eqref{cond2}. In Section 5 we prove our first answer to problem IP1 stated in Theorem \ref{t1} while Section 6 will be devoted to the proof of the different extensions of Theorem \ref{t1} stated in Theorem  \ref{thm_analytic}, \ref{thm_analytic_2}, \ref{t2}, \ref{t3} and \ref{t4}.

\section{Proof of Theorem~\ref{comp_thm}}
\label{sec_comp}
Throughout this section, we assume that $(M,g)$ is a bounded Euclidean domain in $\R^n$, $n\geq 2$. We have the following lemma.
\begin{lem}
	Let $B_R(x_0)$ be the closed ball of radius $R>0$ centred at the point $x_0\in \R^n$. Let $F \in \mathcal C^2(\R)$ satisfy \eqref{F_1} and \eqref{F_2}. Given any $f\in C^{2+\alpha}(\pd B_R(x_0))$, let $u\in C^{2+\alpha}(B_R(x_0))$ be the unique solution to
	\bel{eq_F_u}
	\left\{
	\begin{array}{ll}
		-\Delta u+F(u)=0  & \mbox{in}\ B_R(x_0) ,\\
		u= f&\mbox{on}\ \partial B_R(x_0),
	\end{array}
	\right.
	\ee
 	Then, 
 	$$ |u(x_0)| \leq C,$$
 	for some constant $C>0$ that is independent of $f$.
	\end{lem}

\begin{proof}
	Without loss of generality we set $x_0=0$, $R=1$ and write $B=B_1(0)$. Let $\lambda>0$ be chosen sufficiently large so that 
	$$ |f| \leq \lambda, \quad \forall\, x\in \pd B.$$	
	Using the maximum principle for the semilinear equation \eqref{eq_F_u}, see e.g. \cite[Chapter 10]{GT}, there holds
	$$ v_-(x)\leq u(x)\leq v_+(x) \quad \forall \, x\in B,$$
	where $v_\pm$ is the unique solution to 
	\bel{eq_F_u}
	\left\{
	\begin{array}{ll}
		-\Delta v_\pm+F(v_\pm)=0  & \mbox{in}\ B ,\\
		v_\pm= \pm \lambda &\mbox{on}\ \partial B,
	\end{array}
	\right.
	\ee
	In order to complete the proof of the lemma, it suffices to show that $v_\pm(0)$ is uniformly bounded in $\lambda$. We prove this in the case of $v_+(0)$. The case of $v_-(0)$ follows analogously. For simplicity of notation, we write  $v= v_+$ and note that $v$ must be radial and that it satisfies the following ODE:
	\begin{equation}
		\label{ODE}
		-v''(r) - \frac{n-1}{r}v'(r) + F(v(r)) =0 \quad r\in (0,1),
	\end{equation}
	subject to 
	$$ v'(0)=0 \quad \text{and}\quad v(1)=\lambda>0.$$
	In view of the maximum principle for \eqref{eq_F_u} we deduce that $v(r)$ must be non-negative. Moreover, equation \eqref{ODE} implies that 
	$$ r^{-(n-1)} \frac{d}{d r}(r^{n-1} \frac{d v}{d r}) = F(v(r)) \geq 0.$$
	Thus, there holds, 
	$$ v(r)\geq 0, \quad  v'(r)\geq 0\quad \forall\, r\in [0,1].$$
	Let $G(\mu):= \int_0^\mu F(s)\,ds$, and note that in view of \eqref{F_1} there holds
	\begin{equation}
		\label{G_eq}
		0\leq G(\mu)\leq \mu F(\mu) \quad \forall\, \mu \geq 0. 
	\end{equation}
	Multiplying \eqref{ODE} with $v'(r)$ and integrating from $0$ to $r$, we obtain
	$$
	-\frac{1}{2} v'(r)^2 -\int_0^r\frac{n-1}{s}v'(s)^2\,ds + G(v(r))=0 \quad \forall\, r\in (0,1]. 	 
	$$
	Therefore, 
	\begin{equation}
		\label{v_bound_1}
		0\leq v'(r) \leq \sqrt{2} \sqrt{G(v(r))}\leq \sqrt{2}\, \sqrt{v(r)}\, \sqrt{F(v(r))}.
	\end{equation}
	Plugging the latter bound into \eqref{ODE}, we write
	\bel{v_2der_est} v''(r) \geq \sqrt{F(v(r))}\left(\sqrt{F(v(r))} -\sqrt{2}\frac{n-1}{r}\sqrt{v(r)}\right) \quad \forall\, r\in (0,1).\ee
	In view of \eqref{F_2}, we define $\mu_0>0$ depending only on $\epsilon,\delta>0$ such that
	$$ \frac{F(\mu)}{\mu} \geq 32\,(n-1)^2\quad \forall\, \mu\geq \mu_0,$$
	and note that
	$$\sqrt{F(\mu)} -2\sqrt{2}(n-1)\sqrt{\mu} \geq \frac{1}{2}\sqrt{F(\mu)}\quad \forall\, \mu\geq \mu_0.$$
	Let us assume for now that 
	\begin{equation}\label{assume}
		v(0)>\mu_0.\end{equation}
	Then, it follows from the preceding calculation  together with the fact that $v(r)> \mu_0$ for all $r\in [0,1]$ and \eqref{F_2} that there holds
	$$v''(r) \geq \frac{1}{2} F(v(r))\geq \frac{\delta}{2}\, v(r)^{1+\epsilon} \quad \forall\, r\in [\frac{1}{2},1).$$
	Multiplying the latter equation  with $v'$ and integrating from $r=\frac{1}{2}$ to $r\in (\frac{1}{2},1)$ we write
	$$ v'(r)^2 \geq \frac{\delta}{2+\epsilon}\,(v(r)^{2+\epsilon}- v(\frac{1}{2})^{2+\epsilon})\quad \forall\, r\in (\frac{1}{2},1).$$
	Therefore,
	\bel{v_der_est} \frac{v'(r)}{\sqrt{ v(r)^{2+\epsilon}- v(\frac{1}{2})^{2+\epsilon}}} \geq \sqrt{\frac{\delta}{2+\epsilon}}  \quad \forall\, r\in (\frac{1}{2},1).\ee
	Integrating \eqref{v_der_est} from $r=\frac{1}{2}$ to $r=1$ we deduce that there holds 
	$$ \int_{v(\frac{1}{2})}^{\infty} \frac{1}{\sqrt{s^{2+\epsilon}-v(\frac{1}{2})^{2+\epsilon}}}\,ds> \int_{v(\frac{1}{2})}^{v(1)} \frac{1}{\sqrt{s^{2+\epsilon}-v(\frac{1}{2})^{2+\epsilon}}}\,ds \geq \frac{1}{2}\, \sqrt{\frac{\delta}{2+\epsilon}}.$$
	Observing that
	$$ \int_{v(\frac{1}{2})}^{\infty} \frac{1}{\sqrt{s^{2+\epsilon}-v(\frac{1}{2})^{2+\epsilon}}}\,ds= v(\frac{1}{2})^{-\frac{\epsilon}{2}} \int_1^{\infty} \frac{1}{\sqrt{s^{2+\epsilon}-1}}\,ds,$$
	we conclude that 
	$$ v(\frac{1}{2})^{\epsilon} \leq 2\sqrt{\frac{(2+\epsilon)}{\delta}}\, \int_1^{\infty} \frac{1}{\sqrt{s^{2+\epsilon}-1}}\,ds,.$$ 
	The integral on the right hand side of the above inequality is convergent. Therefore,
	$$v(0)\leq v(\frac{1}{2})\leq C_0,$$ 
	for some constant $C_0$ that is independent of $\lambda$. Recall that we assumed \eqref{assume} to derive the latter bound and therefore we may write in general that
	$$ v(0) \leq \max\{\mu_0,C_0\}.$$
\end{proof}

\begin{lem}
	\label{lem_comp}
	Let $M=B_R(x_0)$ be the closed ball of radius $R>0$ centred at the point $x_0\in \R^n$. Let $a \in \mathcal C^2(M\times \R)$ satisfy \eqref{a_bound_1}--\eqref{a_bound_2} for some $F\in \mathcal C^2(\R)$ that satisfies \eqref{F_1}-\eqref{F_2}. Given any $f\in C^{2+\alpha}(\pd M)$, let $u\in C^{2+\alpha}(M)$ be the unique solution to \eqref{eq1} with Dirichlet data $f$. Then, 
	$$ |u(x_0)| \leq C,$$
	for some constant $C>0$ that is independent of $f$.
\end{lem}

\begin{proof}
	As in the previous lemma, let $\lambda>0$ be sufficiently large so that 
	$$ |f| \leq \lambda, \quad \forall\, x\in \pd B_R(x_0).$$	
	There holds
	$$ u_-(x)\leq u(x)\leq u_+(x) \quad \forall \, x\in B_R(x_0),$$
	where $u_\pm$ is the unique solution to 
	\bel{eq_u_pm}
	\left\{
	\begin{array}{ll}
		-\Delta u_\pm+a(x,u_\pm)=0  & \mbox{in}\ B_R(x_0) ,\\
		u_\pm= \pm \lambda &\mbox{on}\ \partial B_R(x_0),
	\end{array}
	\right.
	\ee
	Let us denote by $v_\pm$ the unique solution to 
	\bel{eq_F_v}
	\left\{
	\begin{array}{ll}
		-\Delta v_\pm+F(v_\pm)=0  & \mbox{in}\ M=B_R(x_0) ,\\
		v_\pm= \pm\lambda &\mbox{on}\ \partial M,
	\end{array}
	\right.
	\ee
	Observe that $v_+\geq 0$ on $M$ while $v_- \leq 0$ on $M$. We claim that
	\begin{equation}
		\label{comp_iden}
		v_-(x) \leq u_-(x) \quad \text{and} \quad u_+(x)\leq v_+(x) \quad \forall\, x\in M.
	\end{equation}
	We prove the bound $u_+\leq v_+$ on $M$. The bound $v_-\leq u_-$ can be proved analogously. Let us start by noting that
	$$ -\Delta v_+ +a(x,v_+) = -F(v_+)+a(x,v_+)\geq 0=-\Delta u_+ + a(x,u_+) \quad \text{on $M$}.$$  
	Therefore,
	$$ -\Delta (u_+-v_+) +(u_+-v_+)\, c\,  \leq 0,\quad \text{on $M$},$$
	where
	$$ c(x) = \int_0^1 \pd_\mu a(x,(1-s) v_+(x)  + su_+(x))\,ds \geq 0\quad \forall\, x\in M.$$
	Using the Maximum principle, we deduce that $u_+\leq v_+$ on $M$. This concludes the proof of \eqref{comp_iden}. Combining the bound \eqref{comp_iden} with the previous lemma, it follows that
	$$ |u(x)| \leq C \quad \forall\, \lambda\geq 0,$$
	for some constant $C>0$ independent of $\lambda$ and thus independent of $f$. 
\end{proof}

We close this section by noting that the proof of Theorem~\ref{comp_thm} follows trivially from the preceding lemma.

\section{Well-posedness of the forward problem}

This section is devoted to the study of the forward problem \eqref{eq1} when the semilinear term $a\in  C^2(M\times\R)$ satisfies condition \eqref{cond1}-\eqref{cond2}. More precisely, we show the well-posedness of \eqref{eq1} for solution that depend continuisly on the boundary data $f$. Our result can be stated as follows.

\begin{prop}\label{p1} Let $a\in  C^2(M\times\R)$ satisfy condition \eqref{cond1}-\eqref{cond2}. Then there exists $\epsilon>0$ depending on  $(M,g)$ and $a$ such that, for $f\in  B_{\epsilon}$, problem \eqref{eq1} admits a unique solution $u_f\in C^{2+\alpha}(M)$ such that $B_{\epsilon}\ni f\mapsto u_f\in C^{2+\alpha}(M)$ is continuous. Moreover, this unique solutions satisfies the estimate
\bel{p1a}\norm{u_f}_{\mathcal C^{2+\alpha}(M)}\leq C\norm{f}_{\mathcal C^{2+\alpha}(\partial M)}.\ee
\end{prop}

\begin{proof} We prove this result by extending the approach of \cite[Theorem B.1.]{CFKKU}, where a similar boundary value problem has been studied, to semilinear terms satisfying \eqref{cond1}-\eqref{cond2}. For this purpose, we introduce the map $\mathcal G$ from $\mathcal C^{2+\alpha}(\partial M)\times\mathcal C^{2+\alpha}(M)$ to the space $\mathcal C^{\alpha}(M)\times\mathcal C^{2+\alpha}(\partial M)$ defined by
$$\mathcal G: (f,u)\mapsto\left(-\Delta_gu+a(x,u), u_{|\partial\Omega}-f\right).$$
 Using the fact that $a\in \mathcal C^2(M\times\R)$ and applying \cite[Theorem  A.8]{H}, we deduce that for any $u\in \mathcal C^{2+\alpha}(M)$ we have $x\mapsto a(x,u(x)), x\mapsto\partial_\mu a(x,u(x))\in \mathcal C^{\alpha}(M)$. Therefore,  the map
$$\mathcal C^{2+\alpha}(M)\ni u\mapsto -\Delta_gu+a(x,u)\in \mathcal C^{\alpha}(M)$$
is $\mathcal C^1$. It follows  that the map $\mathcal G$ is  $\mathcal C^1$ from $\mathcal C^{2+\alpha}(\partial M)\times\mathcal C^{2+\alpha}(M)$ to the space $\mathcal C^{\alpha}(M)\times\mathcal C^{2+\alpha}(\partial M)$. Moreover, \eqref{cond1} implies that $\mathcal G(0,0)=(0,0)$ and
$$\partial_u\mathcal G(0,0)v=\left(-\Delta_g v+\partial_\mu a(x,0)v, v_{|\partial M}\right).$$
Fix $(F,f)\in \mathcal C^{\alpha}(M)\times\mathcal C^{2+\alpha}(\partial M)$ and consider the following linear boundary value problem
\bel{eqa1}
\left\{
\begin{array}{ll}
-\Delta_g v+\partial_\mu a(x,0)v=F  & \mbox{in}\ M ,
\\
v=f &\mbox{on}\ \partial M.
\end{array}
\right.\ee
In light of \eqref{cond1}, problem \eqref{eqa1} admits a unique solution $v\in H^1(M)$. Moreover, using the fact that $\partial_\mu a(\cdot,0)\in C^1(M)$ and applying \cite[Theorem 16.1]{LU} and  
\cite[Theorem 12.1]{LU} we deduce that $v\in \mathcal C^{2+\alpha}( M)$ satisfies the estimate
$$\norm{v}_{\mathcal C^{2+\alpha}(M)}\leq C(\norm{F}_{\mathcal C^{\alpha}(M)}+\norm{f}_{\mathcal C^{2+\alpha}(\partial M)}),$$
with $C>0$ depending only on $a$ and $(M,g)$. Thus, $\partial_u\mathcal G(0,0)$ is an isomorphism from $\mathcal C^{2+\alpha}(M)$ to $\mathcal C^{\alpha}(M)\times\mathcal C^{2+\alpha}(\partial M)$ and, applying the implicit function theorem, we deduce that there exists $\epsilon_1>0$ depending on  $a$, $(M,g)$, and a $\mathcal C^1$ map $\psi$ from $ B_{\epsilon_1}$ to $\mathcal C^{2+\alpha}(M)$, such that, for all $f\in  B_{\epsilon_1}$, we have 
$\mathcal G(f,\psi(f))=(0,0)$.
This proves that, for all $f\in  B_{\epsilon_1}$, $u_f=\psi(f)$ is a solution of \eqref{eq1}. Moreover, using the fact that $\psi$ is $C^1$ from $ B_{\epsilon_1}$ to $\mathcal C^{2+\alpha}(M)$ we deduce that map $f\mapsto u_f$ is also  $C^1$ from $ B_{\epsilon_1}$ to $\mathcal C^{2+\alpha}(M)$. Then, the estimate \eqref{p1a} follows from the fact that $\psi(0)=0$ by definition of the map $\psi$.

The above arguments  prove the existence of solutions of \eqref{eq1} for $f\in  B_{\epsilon_1}$ such that $f\mapsto u_f$ is continuous from $ B_{\epsilon_1}$ to $\mathcal C^{2+\alpha}(M)$. In order to complete the proof of the proposition, we need to  show that there exists $\epsilon\in(0,\epsilon_1]$ such that, for all $f\in  B_{\epsilon}$, such solution is unique. For this purpose, we fix $f\in  B_{\epsilon_1}$ and $u_j\in\mathcal C^{2+\alpha}(M)$, $j=1,2$, solving \eqref{eq1} such that $f\mapsto u_j$ is continuous from $ B_{\epsilon_1}$ to $\mathcal C^{2+\alpha}(M)$. Then, we consider $u=u_1-u_2$ and we notice that $u$ solves  the following linear problem
\bel{eqa2}
\left\{
\begin{array}{ll}
-\Delta_g u+q_fu=0  & \mbox{in}\ M ,
\\
u=0 &\mbox{on}\ \partial M,
\end{array}
\right.\ee
with 
$$q_f(x)=\int_0^1\partial_\mu a(x,su_1(x)+(1-s)u_2(x))ds,\quad x\in M.$$
Using the fact that $a\in  C^2(M\times\R)$ and applying \eqref{p1a}, we deduce that
$$\norm{q_f-\partial_\mu a(\cdot,0)}_{L^\infty(M)}\leq C(\norm{u_1}_{L^\infty(M)}+\norm{u_2}_{L^\infty(M)}).$$
Combining this with the continuity of the map $B_{\epsilon}\ni f\mapsto u_j\in C^{2+\alpha}(M)$, $j=1,2$, we find that 
\bel{p1h}\norm{q_f-\partial_\mu a(\cdot,0)}_{L^\infty(M)}\to 0\textrm{ as }\norm{f}_{\mathcal C^{2+\alpha}(\partial M)}\to 0.\ee
 Now consider the operators $A=-\Delta_g+\partial_\mu a(\cdot,0)$ and $A_f=-\Delta_g+q_f$  with domain $D(A)=D(A_f)=H^1_0(M)\cap H^2(M)$ and consider $\sigma(A)$ (resp. $\sigma(A_f)$) the spectrum of $A$ (resp. $A_f$). In light of \eqref{cond2}, it is clear that $0\not\in\sigma(A)$. Moreover, \eqref{p1h} and the  Min-Max principle imply that   there exists $\epsilon\in(0,\epsilon_1)$ such that 
$$0\not\in \sigma(A_f),\quad f\in B_\epsilon.$$
 This property guaranties the unique solvability of \eqref{eqa2} and, for $f\in B_\epsilon$, we have $u\equiv0$ and $u_1=u_2$.
From this last property we deduce the uniqueness  of solutions $u_f$ of \eqref{eq1} satisfying the continuity property of the map $f\mapsto u_f$.\end{proof}

In view of Proposition \ref{p1}, we know that
$$\{\lambda>0:\ \textrm{for all $f\in B_{\lambda}$ problem \eqref{eq1} is well-posed in $ C^{2+\alpha}( M)$}\}\neq\emptyset.$$
Therefore, we can define $\lambda_1$ given by \eqref{lambda}.

\section{Proof of Theorem \ref{t1}}
We fix $\mu_1>0$ and we assume that \eqref{t1b} is fulfilled. For $\mu\in(-\mu_1,\mu_1)$, we consider $v_\mu$ and $w_\mu$ respectively the solutions of 
$$
\left\{
\begin{array}{ll}
-\Delta_g v_\mu+a(x,v_\mu)=0  & \mbox{in}\ M ,\\
v_\mu= \mu f&\mbox{on}\ \partial M,
\end{array}
\right.$$
$$
\left\{
\begin{array}{ll}
-\Delta_g w_\mu=0  & \mbox{in}\ M ,\\
w_\mu= h_\mu&\mbox{on}\ \partial M.
\end{array}
\right.$$
Fixing $y_\mu=w_\mu -v_\mu$ and applying \eqref{t1b} we deduce that $y_\mu$ satisfies the following condition
\bel{eq4}
\left\{
\begin{array}{ll}
\Delta_g y_\mu=-a(x,v_\mu)  & \mbox{in}\ M ,\\
\partial_\nu y_\mu= 0&\mbox{on}\ \partial M.
\end{array}
\right.
\ee
Integrating this last equation, we deduce that
$$0=\int_{\partial M}\partial_\nu y_\mu d\sigma_g(x)=\int_M\Delta_g y_\mu dV_g(x)=-\int_Ma(x,v_\mu(x))dV_g(x).$$
Therefore, we have
\bel{t1d}\int_Ma(x,v_\mu(x))dV_g(x)=0,\quad \mu\in(-\mu_1,\mu_1).\ee
On the other hand, fixing $v_\mu^{(1)}=\partial_\mu v_\mu$, $\mu\in(-\mu_1,\mu_1)$, we deduce that $v_\mu^{(1)}$ solves the problem
$$\left\{
\begin{array}{ll}
\Delta_g v_\mu^{(1)}+\partial_\mu a(x,v_\mu)v_\mu^{(1)}=0  & \mbox{in}\ M ,\\
v_\mu^{(1)}= f&\mbox{on}\ \partial M
\end{array}
\right.$$
and \eqref{cond3} combined with the maximum principle implies that $v_\mu^{(1)}\geq0 $. Moreover, in light of \eqref{cond1} we have $v_\mu\equiv0$ for $\mu=0$ and we deduce that $v_\mu\geq0$ for $\mu\in[0,\mu_1)$ and $v_\mu\leq0$ for $\mu\in(-\mu_1,0]$. Combining this with \eqref{cond1} and \eqref{cond3} we deduce that
$$a(x,v_\mu(x))\geq 0,\quad (x,\mu)\in \Omega\times[0,\mu_1),$$
$$a(x,v_\mu(x))\leq 0,\quad (x,\mu)\in \Omega\times(-\mu_1,0],$$
which means in particular that, for all $\mu\in(-\mu_1,\mu_1)$, the map $x\mapsto a(x,v_\mu(x))$ is of constant sign and \eqref{t1d} implies that
\bel{t1e}a(x,v_\mu(x))=0,\quad (x,\mu)\in M\times(-\mu_1,\mu_1).\ee
Therefore, $v_\mu$ solves the problem
$$
\left\{
\begin{array}{ll}
-\Delta_g v_\mu=0  & \mbox{in}\ M ,\\
v_\mu= \mu f&\mbox{on}\ \partial M,
\end{array}
\right.$$
and by linearity of this equation we have
$$v_\mu(x)=\mu v_1(x),\quad (x,\mu)\in M\times(-\mu_1,\mu_1)$$ and condition \eqref{t1e} implies
\bel{t1f}a(x,\mu v_1(x))=0,\quad (x,\mu)\in M\times(-\mu_1,\mu_1).\ee
Moreover, the strong maximum principle and condition \eqref{t1a} implies that $v_1\geq c$ which combined with \eqref{t1f} clearly implies \eqref{t1c}. In addition applying \cite[Theorem 3.6]{GT}, we deduce that, for all $\mu\in(-\mu_1,\mu_1)$, $y_\mu$ is constant on $M$ and, for any $x_0\in\partial M$, we have
$$h_\mu(x_0)-\mu f(x_0)=y_\mu(x_0)=y_\mu(x)=h_\mu(x)-\mu f(x),\quad (x,\mu)\in \partial M\times(-\mu_1,\mu_1).$$
This last condition clearly implies \eqref{t1dd}.

\section{Extensions of Theorem \ref{t1}}
This section will be devoted to the proof of the extensions of the results of Theorem \ref{t1} stated in Theorem \ref{thm_analytic}, \ref{thm_analytic_2}, \ref{t2}, \ref{t3} and \ref{t4}.
\subsection{Proof of Theorem~\ref{thm_analytic}}

 Let $m\geq 2$ be the least positive integer such that $F^{(k)}(0)\neq 0$ for $k=m$ while it is zero for all $1\leq k<m$. We write
$$ a(x,s) = q(x) \frac{F^{(m)}(0)}{m!} s^m+q(x) \frac{F^{(m+1)}(0)}{(m+1)!} s^{m+1} + O(|s|^{m+2}),$$
and note in via of the above definition that $F^{(m)}(0)$ is assumed to be non-zero. 

Let $\epsilon \in \R$ be in  a neighborhood of the origin and let $u_\epsilon$ be the unique small solution to \eqref{eq1} subject to Dirichlet boundary data $f= \epsilon$. Then,

$$ u_\epsilon = \epsilon + (\sum_{k=m}^{2m-2}\epsilon^{k} v_k)+ \epsilon^{2m-1} w + O(|\epsilon^{2m}|),$$
where $v_{k}$, $k=m,\ldots, 2m-2$ and $w$ respectively solve the boundary value problems

\bel{eq_b_v}
\left\{
\begin{array}{ll}
	-\Delta_g v_k+q(x) \frac{F^{(k)}(0)}{k!}=0  & \mbox{in}\ M ,\\
	v_k= 0&\mbox{on}\ \partial M,
\end{array}
\right.
\ee

\bel{eq_b_w}
\left\{
\begin{array}{ll}
	-\Delta_g w+q(x) \frac{F^{(2m-1)}(0)}{(2m-1)!}+ \frac{F^{(m)}(0)}{(m-1)!}\,q(x)\,v_m(x)=0  & \mbox{in}\ M ,\\
	w= 0&\mbox{on}\ \partial M.
\end{array}
\right.
\ee
Note also that in view of \eqref{DN_a} there holds
\bel{vw_normal}
\pd_\nu v_k|_{\pd M} = \pd_\nu w|_{\pd M} =0 \quad k=m,\ldots,2m-2.
\ee
Integrating \eqref{eq_b_v} on $M$ we deduce that
\bel{iden_rand_1}
\int_M q(x)\, dV_g = 0.
\ee
Next, integrating \eqref{eq_b_w} on $M$ and using \eqref{iden_rand_1} we deduce that
\bel{iden_rand_2}
\int_M q(x) v_m(x)\,dV_g=0.
\ee 
Returning to \eqref{eq_b_v} with $k=m$ and multiplying this equation with $v_m$ and integrating on $M$ together with \eqref{vw_normal} and \eqref{iden_rand_2} we obtain that
$$ \int_M |\nabla_g v_m|_g^2\, dV_g =0.$$
This implies that $v_m=0$ on $M$ and consequently that $q=0$ on $M$ which implies that $a\equiv 0$. 

\subsection{Proof of Theorem~\ref{thm_analytic_2}}
We give a  proof by contradiction and suppose for contrary that $a(x,s)$ is not identically zero. Let $m\geq 1$ be the smallest integer such that $\pd^{2m}_s a(x,0)$ is not identically zero. In view of the fact that $a(x,s)$ is analytic and even in $s$, we write
$$ a(x,s) = \sum_{k=m}^{\infty} a_{2k}(x)\,s^{2k},$$
for some smooth sequence of functions $\{a_{2k}(x)\}_{k=m}^{\infty}$ on $M$, where $a_{2m}$ is assumed to be not identically zero. Next, let $\epsilon \in \R$ be in  a neighborhood of the origin and let $u_\epsilon$ be the unique small solution to \eqref{eq1} subject to Dirichlet boundary data $f= \epsilon$. Then,

$$ u_\epsilon = \epsilon + (\sum_{k=2m}^{4m-2}\epsilon^{k} v_k)+ \epsilon^{4m-1} w + O(|\epsilon|^{4m}),$$
where $v_{k}$ is zero if $k$ is odd and otherwise solves

\bel{eq_b_v_2}
\left\{
\begin{array}{ll}
	-\Delta_g v_k+ a_{k}=0  & \mbox{in}\ M ,\\
	v_k= 0&\mbox{on}\ \partial M,
\end{array}
\right.
\ee
for $k$ even and in the range $k=2m,\ldots,4m-2$. The function $w$ solves the boundary value problem
\bel{eq_b_w_2}
\left\{
\begin{array}{ll}
	-\Delta_g w+ 2m\,a_{2m}(x)\,v_{2m}(x)=0  & \mbox{in}\ M ,\\
	w= 0&\mbox{on}\ \partial M.
\end{array}
\right.
\ee
Note also that in view of \eqref{DN_a} there holds
\bel{vw_normal_2}
\pd_\nu v_k|_{\pd M} = \pd_\nu w|_{\pd M} =0 \quad k=2m,\ldots,4m-2.
\ee
Integrating \eqref{eq_b_w_2} on $M$ we deduce that
\bel{iden_rand__3}
\int_M a_{2m}(x) v_{2m}(x) dV_g = 0.
\ee
Returning to \eqref{eq_b_v_2} with $k=2m$ together with \eqref{vw_normal_2} the latter identity reduces to 
$$ \int_M |\nabla v_{2m}|^2\, dV_g =0.$$
This implies that $v_{2m}=0$ on $M$ and consequently that $a_{2m}=0$ on $M$ which is a contradiction.

\subsection{Proof of Theorem \ref{t2}}

Fix $\lambda\in(0,\lambda_1]$ and suppose that \eqref{t2b} is fulfilled. For any $\epsilon\in (0,\lambda)$,  $t\in[\epsilon-\lambda,\lambda-\epsilon]$, $s\in [0,\epsilon)$ and $h\in B_1$, we denote by $v_{t,s}$ and $w_{t,s}$ respectively the solutions of 
$$
\left\{
\begin{array}{ll}
-\Delta_g v_{t,s}+a(x,v_{t,s})=0  & \mbox{in}\ M ,\\
v_{t,s}= t +sh&\mbox{on}\ \partial M,
\end{array}
\right.$$
$$
\left\{
\begin{array}{ll}
-\Delta_g w_{t,s}+qw_{t,s}=0  & \mbox{in}\ M ,\\
w_{t,s}= t +sh&\mbox{on}\ \partial M,
\end{array}
\right.$$
In view of \eqref{t4b}, we have
$$\partial_\nu v_{t,s}(x)=\partial_\nu w_{t,s}(x),\quad t\in[\epsilon-\lambda,\lambda-\epsilon],\ s\in [0,\epsilon),\ h\in B_1.$$
Differentiating the above identity with respect to $s$ at $s=0$, we obtain
\bel{t2e} \partial_\nu v_{t}^{(1)}(x)=\partial_\nu w^{(1)}(x),\quad t\in[\epsilon-\lambda,\lambda-\epsilon],\ s\in [0,\epsilon),\ h\in B_1,\ x\in\partial M,\ee
with $v_{t}^{(1)}=\partial_{s}v_{t,s}|_{s=0}$ and  $w^{(1)}=\partial_{s}w_{t,s}|_{s=0}$. On the other hand, by the first order linearization, one can check that   $v_{t}^{(1)}$ and $w^{(1)}$ solve respectively the problems
\bel{t2f}
\left\{
\begin{array}{ll}
-\Delta_g v_{t}^{(1)}+\partial_\mu a(x,v_{t,0})v_{t}^{(1)}=0  & \mbox{in}\ M ,\\
v_{t}^{(1)}= h&\mbox{on}\ \partial M,
\end{array}
\right.\ee

$$
\left\{
\begin{array}{ll}
-\Delta_g w^{(1)}=0  & \mbox{in}\ M ,\\
w^{(1)}= h&\mbox{on}\ \partial M.
\end{array}
\right.$$
Fixing $y=w^{(1)}-v_{t}^{(1)}$, with $t=0$ and $h\equiv1$, and applying \eqref{t2e}, we obtain
\bel{t2f}\left\{
\begin{array}{ll}
\Delta_g y=-\partial_\mu a(x,v_{0,0})v_{0}^{(1)}  & \mbox{in}\ M ,\\
y=\partial_\nu y= 0&\mbox{on}\ \partial M.
\end{array}
\right.\ee
Integrating this formula, we get 
$$0=\int_{\partial M}\partial_\nu y d\sigma_g(x)=\int_M\Delta_g y dV_g(x)=-\int_M\partial_\mu a(x,v_{0,0})v_{0}^{(1)}(x)dV_g(x).$$
In addition,  condition \eqref{cond1} implies that $v_{0,0}\equiv0$ and we get
\bel{t2g}\int_M\partial_\mu a(x,0)v_{0}^{(1)}(x)dV_g(x)=0.\ee
Applying the maximum principle, we deduce that $v_{0}^{(1)}\geq0$ and, applying \eqref{t2a}, we obtain
$$\partial_\mu a(x,0)v_{0}^{(1)}(x)=0,\quad x\in M.$$
Moreover, the Harnack inequality and unique continuation results for elliptic equations imply that, for all $x\in M$, $v_{0}^{(1)}(x)>0$ and it follows that $\partial_\mu a(\cdot,0)\equiv0$. Thus, we have 
$$\lim_{t\to0}\norm{\partial_\mu a(\cdot,v_{t,0}(\cdot))}_{L^\infty(M)}=0$$
and applying  \eqref{t2e} and \cite[Theorem 2]{CFO}, we deduce that there exists $\lambda\in[0,\lambda_1)$ such that 
$$ \partial_\mu a(x,v_{t,0}(x))=0\quad  t\in[\epsilon-\lambda,\lambda-\epsilon]\  x\in M,$$
and consequently that 
$$a(x,v_{t,0}(x))=0\quad  t\in[\epsilon-\lambda,\lambda-\epsilon]\ x\in M.$$
Thus, $v_{t,0}$, $t\in[\epsilon-\lambda,\lambda-\epsilon]$ solves the problem
$$
\left\{
\begin{array}{ll}
-\Delta_g v_{t,0}=0  & \mbox{in}\ M ,\\
v_{t,0}= t &\mbox{on}\ \partial M,
\end{array}
\right.$$
which implies that $v_{t,0}\equiv t$. Therefore, we have
$$a(x,t)=0\quad  t\in[\epsilon-\lambda,\lambda-\epsilon]\ x\in M$$
and sending $\epsilon\to0$, we obtain \eqref{t2c}.
\subsection{Proof of Theorem \ref{t3}}

Fix $\lambda\in(0,\lambda_1]$ and suppose that \eqref{t2b} is fulfilled. For any   $t\in[-\lambda,\lambda]$ and $h\in B_1$ we denote by $v_{t,h}$ and $w_{t,h}$ respectively the solutions of 
$$
\left\{
\begin{array}{ll}
-\Delta_g v_{t,h}+a(x,v_{t,h})=0  & \mbox{in}\ M ,\\
v_{t,h}= th&\mbox{on}\ \partial M,
\end{array}
\right.$$
$$
\left\{
\begin{array}{ll}
-\Delta_g w_{t,h}+qw_{t,h}=0  & \mbox{in}\ M ,\\
w_{t,h}= th&\mbox{on}\ \partial M,
\end{array}
\right.$$
In view of \eqref{t2b}, we have
$$\partial_\nu v_{t,h}(x)=\partial_\nu w_{t,h}(x),\quad t\in[-\lambda,\lambda],\ h\in B_1,\ x\in\partial M.$$
Differentiating the above identity with respect to $t$, we obtain
\bel{t3e} \partial_\nu v_{t,h}^{(1)}(x)=\partial_\nu w_h^{(1)}(x),\quad t\in[-\lambda,\lambda],\ \ h\in B_1,\ x\in\partial M,\ee
with $v_{t,h}^{(1)}=\partial_tv_{t,h}$ and  $w_h^{(1)}=\partial_tw_{t,h}$. On the other hand, by the first order linearization, one can check that $v_{t,h}^{(1)}$ and $w_h^{(1)}$ are respectively the solutions of 
$$
\left\{
\begin{array}{ll}
-\Delta_g v_{t,h}^{(1)}+\partial_\mu a(x,v_{t,h})v_{t,h}^{(1)}=0  & \mbox{in}\ M ,\\
v_{t,h}^{(1)}= h&\mbox{on}\ \partial M,
\end{array}
\right.$$
$$
\left\{
\begin{array}{ll}
-\Delta_g w_h^{(1)}+qw_h^{(1)}=0  & \mbox{in}\ M ,\\
w_h^{(1)}= h&\mbox{on}\ \partial M.
\end{array}
\right.$$
In view of \cite[Theorem 3]{DKSU}, condition \eqref{t3e} implies that
\bel{t3f}\partial_\mu a(x,v_{t,h}(x))=q(x),\quad x\in M.\ee
In particular $w_h^{(1)}$ and $v_{t,h}^{(1)}$ solve the same  boundary value problem and the uniqueness of such solutions, guaranteed by condition \eqref{cond2}, implies that $w_h^{(1)}=v_{t,h}^{(1)}$. Then, condition \eqref{cond1} implies that
$$v_{t,h}=\int_0^tv_{\tau,h}^{(1)}d\tau=t w_h^{(1)}$$
and from \eqref{t3f} we deduce that
\bel{t3g}\partial_\mu a(x,t w_h^{(1)}(x))=q(x),\quad x\in M.\ee
Since $q$ is non-negative, fixing $h\equiv1$ and applying the  maximum principle, the Harnack inequality and unique continuation results for elliptic equations, we deduce that 
$$c=\inf_{x\in M}w_h^{(1)}(x)>0.$$
Combining this with \eqref{t3g}, we deduce \eqref{t2c}.

\subsection{Proof of Theorem \ref{t4}}

Fix $\lambda\in(0,\lambda_1]$ and suppose that \eqref{t4b} is fulfilled. For any $\epsilon\in (0,\lambda)$,  $t\in[\epsilon-\lambda,\lambda-\epsilon]$, $(s_1,s_2)\in [0,\epsilon/2)^2$ and $h_1,h_2\in B_1$ we fix $s=(s_1,s_2)$ and we denote by $v_{t,s}$ and $w_{t,s}$ respectively the solutions of 
$$
\left\{
\begin{array}{ll}
-\Delta_g v_{t,s}+a_2(x,v_{t,s})=0  & \mbox{in}\ M ,\\
v_{t,s}= t +s_1h_1+s_2h_2&\mbox{on}\ \partial M,
\end{array}
\right.$$
$$
\left\{
\begin{array}{ll}
-\Delta_g w_{t,s}+qw_{t,s}=0  & \mbox{in}\ M ,\\
w_{t,s}= t +s_1h_1+s_2h_2&\mbox{on}\ \partial M.
\end{array}
\right.$$
In view of \eqref{t4b}, we have
$$\partial_\nu v_{t,s}(x)=\partial_\nu w_{t,s}(x),\quad t\in[\epsilon-\lambda,\lambda-\epsilon],\ s\in [0,\epsilon/2)^2,\ h_1,h_2\in B_1,\ x\in\partial M.$$
Differentiating the above identity with respect to $s_1$ and $s_2$ at $s=0$, we obtain
\bel{t4e} \partial_\nu v_{t}^{(2)}(x)=\partial_\nu w_{h}^{(2)}(x),\quad t\in[\epsilon-\lambda,\lambda-\epsilon],\ s\in [0,\epsilon),\ h_1,h_2\in B_1,\ x\in\partial M,\ee
with $v_{t}^{(2)}=\partial_{s_1}\partial_{s_2}v_{t,s}|_{s=0}$ and  $w_{h}^{(2)}=\partial_{s_1}\partial_{s_2}w_{t,s}|_{s=0}$. On the other hand, by the second order linearization, one can check that $w_{h}^{(2)}\equiv0$ and $v_{t,h}^{(2)}$ solves the problem
\bel{t2f}
\left\{
\begin{array}{ll}
-\Delta_g v_{t}^{(2)}+\partial_\mu a_2(x,v_{t,0})v_{t,h}^{(2)}=-\partial_\mu^2 a_2(x,v_{t,0,h})v_{1,t}^{(1)}v_{2,t}^{(1)}  & \mbox{in}\ M ,\\
v_{t}^{(2)}= 0&\mbox{on}\ \partial M,
\end{array}
\right.\ee
with $v_{j,t}^{(1)}$, $j=1,2$, the solution of
$$
\left\{
\begin{array}{ll}
-\Delta_g v_{j,t}^{(1)}+\partial_\mu a_2(x,v_{t,0})v_{j,t}^{(1)}=0  & \mbox{in}\ M ,\\
v_{j,t}^{(1)}= h_j&\mbox{on}\ \partial M,
\end{array}
\right.$$
In view of \eqref{t4e}, we have
$$\partial_\nu v_{t}^{(2)}(x)=0,\quad t\in[\epsilon-\lambda,\lambda-\epsilon],\ s\in [0,\epsilon),\ h_1,h_2\in B_1,\ x\in\partial M.$$
Fixing $v_{3,t}^{(1)}\in H^2(M)$ an arbitrary solution of $-\Delta_g v_{3,t}^{(1)}+\partial_\mu a_2(x,v_{t,0})v_{3,t}^{(1)}=0$,  multiplying \eqref{t2f} by $v_{3,t}^{(1)}$ and integrating by parts, we find
\bel{t4g} \int_M \partial_\mu^2 a_2(x,v_{t,0}(x))v_{1,t}^{(1)}(x)v_{2,t}^{(1)}(x)v_{3,t}^{(1)}(x)dV_g(x)=0.\ee
Using the fact that in the above identity $v_{j,t}^{(1)}$, $j=1,2$, can be seen as arbitrary smooth solutions of the equations $-\Delta_g v+\partial_\mu a_2(x,v_{t,0})v=0$, we are in position to apply \cite[Proposition 6]{FO20} in order to deduce that
$$ \partial_\mu^2 a_2(x,v_{t,0}(x))=0,\quad t\in[\epsilon-\lambda,\lambda-\epsilon],\ x\in  M.$$
Sending $\epsilon\to0$, we deduce that
\bel{t4g}\partial_\mu^2 a_2(x,v_{t,0}(x))=0,\quad t\in[-\lambda,\lambda],\ x\in  M.\ee
Recalling that $ w_t=\partial_t^2v_{t,0}$ solves the problem 
$$\left\{
\begin{array}{ll}
-\Delta_g w_t+\partial_\mu a_2(x,v_{t,0})w_t=-\partial_\mu^2 a_2(x,v_{t,0})(\partial_tv_{t,0,h})^2  & \mbox{in}\ M ,\\
w_{t}= 0&\mbox{on}\ \partial M,
\end{array}
\right.$$
and applying \eqref{t4g}, we deduce that $w_t\in H^1_0(M)\cap H^2(M)$ satisfies \bel{t4cc}-\Delta_g w_t(x)+\partial_\mu a_2(x,v_{t,0}(x))w_t(x)=0,\quad x\in M.\ee In addition, fixing the operators $A=-\Delta_g+\partial_\mu a_2(\cdot,0)$ and $A_t=-\Delta_g+\partial_\mu a_2(\cdot,v_{t,0})$  with domain $D(A)=D(A_t)=H^1_0(M)\cap H^2(M)$ and recalling that 
$$\begin{aligned}&\limsup_{t\to0}\norm{\partial_\mu a_2(\cdot,v_{t,0})-\partial_\mu a_2(\cdot,0)}_{L^\infty(M)}\\
&\leq \limsup_{t\to0}\underset{\mu\in[-\norm{v_{t,0}}_{L^\infty(M)},\norm{v_{t,0}}_{L^\infty(M)}]}{\sup}\norm{\partial_\mu^2 a_2(\cdot,\mu)}_{L^\infty(M)}\norm{v_{t,0}}_{L^\infty(M)}=0,\end{aligned}$$ 
we can deduce from the Min-Max principle and \eqref{t4aa}  that   there exists $\lambda\in(0,\lambda_1]$ such that 
$$0\not\in \sigma(A_t),\quad t\in(-\lambda,\lambda).$$
Combining this with \eqref{t4cc} we deduce that $\partial_t^2v_{t,0}=w_t\equiv 0$, $t\in[-\lambda,\lambda]$. It follows that
$$\partial_tv_{t,0}=\partial_tv_{0,0},\quad t\in[-\lambda,\lambda].$$
Then, \eqref{cond1} implies 
\bel{t4h}v_{t,0}=t\partial_tv_{0,0}.\ee
Recall that $y=\partial_tv_{0,0}$ solves the boundary value problem
$$\left\{
\begin{array}{ll}
-\Delta_g y+\partial_\mu a_2(x,0)y=0  & \mbox{in}\ M ,\\
y= 1&\mbox{on}\ \partial M.
\end{array}
\right.$$
Therefore, condition \eqref{t4aa} combined with the  maximum principle, the Harnack inequality and unique continuation results for elliptic equations imply that 
$$c=\inf_{x\in M}y(x)>0.$$
Combining this with \eqref{t4g}-\eqref{t4h}, we deduce \eqref{t4c}.

\bigskip
\vskip 1cm

\end{document}